\documentclass[12points,reqno]{amsart}
\usepackage{amsfonts}
\usepackage{amssymb}
\usepackage{graphicx}
\usepackage{amsmath}

\newtheorem{theorem}{Theorem}
\newtheorem{lemma}{Lemma}

\theoremstyle{definition}

\theoremstyle{remark}

\numberwithin{equation}{section}

\everymath{\displaystyle}

\begin{document}
\title{ Finite presentation is a Morita invariant property }

%

\author{Adel Alahmadi}
\address{Department of Mathematics, King Abdulaziz University, P.O.Box 80203, Jeddah, 21589, Saudi Arabia}
\email{analahmadi@kau.edu.sa}
\author{Hamed Alsulami}
\address{Department of Mathematics, King Abdulaziz University, P.O.Box 80203, Jeddah, 21589, Saudi Arabia}
\email{hhaalsalmi@kau.edu.sa}
%
%


\maketitle

\begin{abstract}
We prove that the property  of an algebra to be finitely presented is Morita invariant.
\end{abstract}

\maketitle

\section{Introduction}

Let $F$ be a field. Throughout the paper we consider associative unital $F-$ algebras. Two algebras $A$ and $B$ are called
Morita equivalent if their categories of left modules are equivalent. We say that a property $P$ is Morita invariant if any two Morita equivalent algebras do  satisfy or do not satisfy $P$  simultaneously.

\medskip

In \cite{MS} S. Montgomery and L. Small proved that finite generation is a Morita invariant property. In this note we show that the property of having a finite presentation is also Morita invariant.

\section{ Main Result}

\begin{theorem}\label{thm1} Finite presentation is a Morita invariant property.

\end{theorem}

\medskip

\par Let's recall some definitions.
Let $A$ be an $F$- algebra generated by a finite collection of elements $a_1,\cdots ,a_m.$ Consider the free associative algebra $F\langle x_1,\cdots, x_m\rangle$
and the homomorphism  $$F\langle x_1,\cdots, x_m\rangle \xrightarrow{\varphi} A, x_i \rightarrow a_i, 1 \leq i \leq m.$$ We say that the algebra $A$ is finitely presented (f.p.) if the ideal $I= \ker\varphi$ is finitely generated as an ideal. This property does not depend on a choice of a generating system of $A$ as long as this system is finite. If the ideal $I$ is generated by elements $f_1(x_1,\cdots, x_m), \cdots, f_s (x_1,\cdots, x_m)$ then we say that the algebra $A$ has presentation $$A=\langle x_1,\cdots, x_m | f_1=0,\cdots, f_s=0\rangle.$$

It is easy to see that the ideal $I$ is not finitely generated if and only it is a union of a strictly ascending chain of ideals $I_1\subset I_2\subset\cdots$ of the algebra $F\langle x_1,\cdots, x_m\rangle$. Equivalently, the algebra $A$ is not finitely presented if and only if there exists a sequence of $m$- generated algebras $A_i$ and epimorphisms $A_i \xrightarrow{\varphi_{i}} A_{i+1}, \ker \varphi_i\neq (0), i\geq 1$ , such that $A$ is the inductive limit of this sequence.

More precisely, consider the subalgebra $\hat{A}=\{\{a_i\}_{i\geq1} | a_i\in A_i, a_{i+1}=\varphi_i(a_i),i\geq 1\}$ of the Cartesian product $\sqcap_{i\geq 1}A_i.$ Let $J$ be the ideal of the algebra $\hat{A}$ that consists of sequences $\{a_i\}_{i\geq1}$ such that $a_i=0$ for all sufficiently large $i.$ Then $A\cong \hat{A}/J.$

\medskip

\begin{lemma}\label{lem1}
Let $A$ be a finitely generated algebra and let $M_n(A)$ be the algebra of $n \times n$ matrices over $A$. The algebra $M_n(A)$ is f.p. if and only if the algebra $A$ is f.p.
\end{lemma}

\begin{proof}
Let $A=\langle x_1,\cdots, x_m | f_1 (x_1,\cdots, x_m)=0,\cdots, f_s (x_1,\cdots, x_m)=0\rangle $ be a finite presentation of the algebra $A$. Then the matrix algebra $M_n(A)$ has a presentation

$$M_n(A)=\langle x_1,\cdots, x_m, y_{ij}, 1 \leq i,j \leq n\hspace{.1cm} |\hspace{.1cm} y_{ij}y_{kl}=\delta_{jk}y_{il},$$ $$x_i=y_{11}x_{i}~y_{11}, f_t(x_1,\cdots, x_m)=0, 1\leq i,j,k,l \leq n, 1 \leq t \leq s \rangle.$$

\medskip

Suppose on the other hand that the algebra $A$ is not finitely presented. Then there exists a sequence of $m$- generated algebras $A_i$ and epimorphisms \newline $A_i \xrightarrow{\varphi_{i}} A_{i+1},\ker\varphi_i \neq (0), i \geq 1,$ with $A$ as the inductive limit.

An epimorphism $\varphi_i$ gives rise to the epimorphism $M_n(A_i)\rightarrow M_n(A_{i+1})$ with nonzero kernel. The inductive limit of this system of homomorphisms is $M_n(A)$ which implies that the algebra $M_n(A)$ is not finitely presented. This completes the proof of the Lemma.

\end{proof}

\medskip

An idempotent $e$ of an algebra $A$ is said to be full if $AeA=A$. It is known (see \cite{L}, prop.18.33) that algebras $A, B$ are Morita equivalently if and only if $B \cong eM_n(A)e$, where $n \geq 1$ and $e$ is a full idempotent of the matrix algebra $M_n(A)$.

In \cite{AA}, we proved that if $A$ is a f. p. algebra and $e, 1-e \in A$ are full idempotents then the Peirce component $eAe$ is f.p.

If we drop the assumption that the idempotent $1-e$ is full then Theorem $1$ follows immediately. We will drop it first for the particular case when $eAe$ is isomorphic to the algebra of $n\times n$ matrices over some algebra, $n\geq 2$.
\medskip

\begin{lemma}\label{lem2}
Let $A$ be a f.p. algebra with a full idempotent $e$. Suppose that $B=eAe\cong M_n(C), n\geq 2$, for some associative algebra $C$. Then the algebra $B$ is f. p.
\end{lemma}

\begin{proof}
Let $e_{ij}, 1 \leq i,j \leq n$, be matrix units of the algebra $B$, $\sum\limits_{i=1}^{n }e_{ii}=e$. Consider the idempotent $e_{11}$. We have $Ae_{11}A \supseteq Be_{11}B=B$. This implies that the idempotent $e_{11}$ is full. On the other hand $(1-e_{11})e_{22}=e_{22},$ hence $ A(1-e_{11})A \supseteq Be_{22}B=B$, which implies that the idempotent $1-e_{11}$ is full as well. By [AA] the algebra $e_{11}Ae_{11} \cong C$ is f.p. Hence by Lemma 1 the algebra $B$ is f. p. as well, which completes the proof of the Lemma.

\end{proof}

\medskip

\begin{proof}[ Proof of Theorem \ref{thm1}]
Let $A, B$ be Morita equivalent algebras. Suppose that the algebra $A$ is f. p. The algebras $A, B^{'}=M_{2}(B)$ are also Morita equivalent. There exists $n \geq 1$ and a full idempotent $e$ of the matrix algebra $M_n(A)$ such that $B^{'}\cong e M_n(A^{'}e)$. By Lemma 2 the algebra $B^{'}$ is f. p. Hence by Lemma 1 the algebra $B$ is f. p. as well, which completes the proof of theorem 1.
\end{proof}

\section*{Acknowledgement}
This project was funded by the Deanship of Scientific Research (DSR), King Abdulaziz University, under Grant No.
(27-130-36-HiCi). The authors, therefore, acknowledge technical and financial support of KAU.

\end{document}